\documentclass[12pt,a4paper,leqno]{article}
\usepackage{amsmath,amsthm,amssymb}

\newtheorem{theorem}{Theorem}[section]

\newtheorem{corollary}{Corollary}[section]
\newtheorem{proposition}{Proposition}[section]

\theoremstyle{definition}{%

\newtheorem{example}{Convention}[section]
\newtheorem{remark}{Remark}[section]}

\DeclareMathOperator{\GL}{GL}
\newcommand{\itl}[1]{{\em#1}}
\newcommand{\ffq}{\mathbb F_q}

\newcommand{\BS}{blocking set}
\newcommand{\PG}[2]{\mbox{{\rm PG}$(#1,#2)$}}
\newcommand{\notpar}{\mathbin{\not{}\hspace{-0.35ex}{\parallel}\hspace{0.35ex}}}

\begin{document}

\author{Andrea Blunck\footnotemark[1]\qquad
Hans Havlicek\footnotemark[2]\qquad Corrado
Zanella\footnotemark[3]}

\footnotetext[1]{Fachbereich Mathematik, Universit{\"a}t Hamburg,
Bundesstra{\ss}e 55, D-20146 Hamburg, Germany. e-mail \tt
andrea.blunck@math.uni-hamburg.de}

\footnotetext[2]{Hans Havlicek, Institut f{\"u}r Diskrete
Mathematik und Geometrie, Technische Universit{\"a}t Wien, Wiedner
Hauptstra{\ss}e 8--10, A-1040 Wien, Austria. e-mail \tt
havlicek@geometrie.tuwien.ac.at }

\footnotetext[3]{Corrado Zanella, Dipartimento di Tecnica e
Gestione dei Sistemi Industriali, Universit{\`a} di Padova,
Stradella S. Nicola, 3, I-36100 Vicenza, Italy. e-mail \tt
corrado.zanella@unipd.it }

\date{}

\title{Blocking sets in chain geometries}
\maketitle

\section{Introduction}

Let $R$ be a ring\footnote{The ring $R$ is assumed to be associative, with a unit element $1$,
which is inherited by subrings. The trivial case $1=0$ is excluded.}.
The \itl{projective line over $R$}, denoted by $ \mathbb{P}(R)$, is the set of all submodules of $R^2$ of type $R(a,b)$,
where $(a\ b)$ is the first row of some invertible $2\times 2$ matrix over $R$.

Suppose now that a field $K$ is contained in $R$, as a subring.
The (\itl{generalized}) \itl{chain geometry} associated with $K$ and $R$, denoted by $\Sigma(K,R)$, is the structure whose \itl{points} are the elements of $ \mathbb{P}(R)$ and whose blocks
(called \itl{chains}) are the sets $ \mathbb P(K)^g$ with $g\in\GL_2(R)$.
Here $ \mathbb P(K)$ is embedded in $ \mathbb P(R)$ by means of $K(a,b)\hookrightarrow R(a,b)$.
Roughly speaking, the chains are projective lines over $K$ contained in the projective line over $R$.

The classical example of a chain geometry is $\Sigma(\mathbb R,\mathbb C)$, or, by generalizing a little,
$\Sigma(K,R)$ where $R$ is a field and $[R:K]=2$.
In this case $\Sigma(K,R)$ is usually called \itl{Miquelian M\"obius plane}.

Two points $R(a,b)$ and $R(c,d)$ in $ \mathbb P(R)$ are called \itl{distant}, in symbols
\newline$R(a,b)\triangle R(c,d)$,
when $\begin{pmatrix}a&b\\ c&d\end{pmatrix}\in\GL_2(R)$.
We have $R(a,b)\triangle R(c,d)$ if, and only if, both points are on a common chain.
The group $\GL_2( R)$ acts transitively on the set of all triples of mutually distant points.

A \itl{\BS}\ in a geometry of points and blocks is a set, say $B$, of points, such that
every block contains at least
one element of $B$.
The most investigated question regarding the \BS s concerns their minimum size.
In this paper we give some basic results on this problem for a  finite chain geometry.
More precisely, in Section \ref{examples}, two examples of chain geometries are given.
Section \ref{sizes} is concerned with the number $\lambda_i$ of blocks containing $i$ given
mutually distant points, $i=0,1,2,3$.
In Section \ref{bounds},
lower bounds for the size of a \BS\ in $\Sigma(K,R)$ are given both
in the general case (see (\ref{elf}))
and in case $R$ is a local ring (Theorem \ref{principale}).
Two examples attaining the general lower bound are exhibited.
It is also shown that it is possible to construct \BS s in chain geometries, starting from a
\BS\  in a M\"obius geometry (Theorem \ref{ponte}).

\section{Examples of chain geometries}\label{examples}

We give a short description of two classes of chain geometries we will deal with in this paper.

\begin{example}
Let $R$ be the direct product $K\times K$.
Then $R$ has precisely two nontrivial ideals: $K(1,0)$ and $K(0,1)$.
The ring $R$ becomes a $K$-algebra via the embedding $x\hookrightarrow(x,x)$ of $K$ into $R$.
A submodule $R(a,b)$ of $R^2$ is a point if and only if $a$ and $b$ do not belong to a common nontrivial ideal.
Let $x_0$, $x_1$, $x_2$ and $x_3$ denote the homogeneous coordinates in PG$(3,K)$.
For $a,b\in R$ set $a=(a_1,a_2)$, $b=(b_1,b_2)$.
The map $\psi:\mathbb P(R)\rightarrow\PG 3K$ defined by
\[
  R(a,b)^\psi=K(a_1b_2,a_2b_1,a_1a_2,b_1b_2)
\]
is a bijection between $\mathbb P(R)$ and the hyperbolic quadric $\cal Q$ in \PG 3K\ of equation $x_0x_1-x_2x_3=0$.
The image of $\mathbb P(K)$ under $\psi$ is the intersection of $\cal Q$ with the plane $x_0-x_1=0$.
Since a plane of equation $u_0x_0+u_1x_1+u_2x_2+u_3x_3=0$ is tangent to $\cal Q$ if and only if $u_0u_1-u_2u_3=0$, we
see that $\mathbb P(K)^\psi$ is a nondegenerate conic.
On the other hand, the points which are non-distant from $R(1,0)$ are precisely those on the tangent
plane section given by $x_3=0$.
Next, the mapping $\varphi:\,(\GL_2(K))^2\rightarrow\GL_2(R)$ defined by setting, for every $M,N\in\GL_2(K)$,
\[
  {((a_1,a_2)\ (b_1,b_2))}\,{(M,N)^\varphi}=((\alpha_1,\alpha_2)\ (\beta_1,\beta_2)),
\]
\[
  (\alpha_1\ \beta_1)=(a_1\ b_1)M, \quad (\alpha_2\ \beta_2)=(a_2\ b_2)N,
\]
is an isomorphism between the direct product of $\GL_2(K)$ with itself  and $\GL_2(R)$.
For $M=\begin{pmatrix}m_1&m_2\\ m_3&m_4\end{pmatrix}\in\GL_2(K)$, the action of $(M,1)^\varphi$ on $\cal Q$
is the restriction of the projectivity of PG$(3,K)$ defined by
\[
  (x_0\ x_1\ x_2\ x_3)\ \mapsto\ (x_0\ x_1\ x_2\ x_3)\,
  \begin{pmatrix} m_1&0&0&m_2\\ 0&m_4&m_3&0\\ 0&m_2&m_1&0\\ m_3&0&0&m_4\end{pmatrix},
\]
that fixes $\cal Q$.
A similar property is satisfied by $(1,N)^\varphi$ where $N\in\GL_2(K)$.
It easily follows that (\itl i) the action of each element of $\GL_2(R)$ on $\cal Q$ can be represented by an element
of PGL$(4,K)$ fixing $\cal Q$, (\itl{ii}) the images of the chains under the embedding $\psi$ are precisely
the nondegenerate conics contained in $\cal Q$, and (\itl{iii}) two distinct points $p,q\in\mathbb P(R)$
are distant if,
and only if, the line through $p^\psi$ and $q^\psi$ is not contained in $\cal Q$.

More generally, if $R$ is a kinematic algebra, i.e.\ for each $x\in R$ two elements $k,l\in K$ exist such that
$x^2=kx+l$, and  $K\neq\mathbb F_2$, then the points of $\Sigma(K,R)$ can be represented as points of a
quadric ${\cal Q}'$ in a projective space over $K$ of suitable dimension,
and distant points correspond to points that are not conjugate with respect to ${\cal Q}'$ \cite{Hot76}.
See also \cite[Section 6.2]{Herzer}.
\end{example}

\begin{example}
Let $R$ be a local ring, and let $R^*$ be the set of all units in $R$.
Each point, say $R(a,b)$, of the projective line $\mathbb P(R)$ has the property that at least one
of the two elements $a$, $b$ is invertible.
Because since $R\setminus R^*$ is an ideal the existence of an inverse matrix
$\begin{pmatrix} x&*\\ y&*\end{pmatrix}$ would otherwise lead to the contradiction $1=ax+by\in R\setminus R^*$.
So $\mathbb P(R)$ is the disjoint union
\begin{equation}\label{disjunion}
  \mathbb P(R)=\left\{R(x,1)|\,x\in R\}\cup\{R(1,z)|\,z\in R\setminus R^*\right\}.
\end{equation}
In this case the complementary relation of $\triangle$, which we will denote by $\|$ (\itl{parallelism}), is an equivalence relation.
More explicitly, this means for arbitrary $x,y\in R$, $z,w\in R\setminus R^*$:
\begin{equation}\label{morexpl}
  R(1,z)\|R(1,w);\ R(x,1)\notpar R(1,z);\ \left(R(x,1)\|R(y,1)\Leftrightarrow x-y\in R\setminus R^*\right).
\end{equation}
Using the description in (\ref{morexpl}) one can easily see that $\|$ in fact is an equivalence relation.
\end{example}

\section{Finite chain geometries}\label{sizes}

From now on we assume that $R$ is finite. So, $K=\mathbb F_q$, $q$ a prime
power, and $R$ is in a natural way a left vector space over $\mathbb F_q$.
Define $d=\dim_{\mathbb F_q}R$. Since $\GL_2( R)$ acts transitively on the
triples of mutually distant points, the number of chains containing $i$ given
mutually distant points, $i=0,1,2,3$, is a constant, say $\lambda_i$. The
problem to determine the numbers $\lambda_0$, $\ldots$, $\lambda_3$ is
intricate. However, to our purposes it is enough to describe their ratios.
\begin{proposition}
  Let $v$ be the number of points of $\Sigma(\mathbb F_q,R)$.
  Denote by $R^*$ the set of units of $R$, and let $\#R=q^d$, $\#R^*=r^*$.
  Then
  \begin{eqnarray}
    \label{ela}\lambda_0&=&\frac{vq^{d-1}r^*}{q^2-1}\lambda_3;\\
    \label{elb}\lambda_1&=&\frac{q^{d-1}r^*}{q-1}\lambda_3;\\
    \label{elc}\lambda_2&=&\frac{r^*}{q-1}\lambda_3.
  \end{eqnarray}
\end{proposition}

\begin{proof}
  The points which are distant from $R(1,0)$ are precisely those in the
form $R(a,1)$, $a\in R$; since they are all distinct and $\GL_2( R)$
acts transitively on $\mathbb P(R)$,
we have that each point in $\mathbb P(R)$ is distant from precisely $q^d$ points.
Similarly, the points which are distant from both $R(1,0)$ and $R(0,1)$ form the set $\{R(1,a)|\,a\in R^*\}$.
There are $r^*$ of such points.

Assume now that $p_1$ and $p_2$ are two distant points.
Each chain has precisely $q+1$ points.
So, by counting in two ways
the number $M$ of pairs $(p,C)$, where $p$ is a point distant from both $p_1$ and $p_2$, and $C$ is a chain through $p$, $p_1$ and $p_2$, we obtain $M=r^*\lambda_3=(q-1)\lambda_2$ and this gives (\ref{elc}).
A similar argument yields (\ref{elb}) and (\ref{ela}).
\end{proof}

For sake of completeness we mention that $\lambda_3=r^*/\#N$, where $N=\{n\in R^*|\,n^{-1}K^*n=K^*\}$ is the normalizer
of $K^*$ in $R^*$.
See e.g. \cite{Ha04}.

Since all points described in (\ref{disjunion}) are in $\mathbb P(R)$, even if $R$ is not local, we see that in general
\begin{equation}\label{eld}
  v\ge2q^d-r^*.
\end{equation}

\section{Blocking sets}\label{bounds}

A \itl{blocking set}\ in $\Sigma(\ffq,R)$ is a set $B$ of points, such that every chain contains at least one element of $B$.
A trivial lower bound for the size of $B$, holding in each geometry where the number of blocks through a point is a constant, is
\begin{equation}\label{ele}
  \#B\ge\frac{\lambda_0}{\lambda_1},
\end{equation}
whence, by (\ref{ela}), (\ref{elb}), and (\ref{eld}),
\begin{equation}\label{elf}
  \#B\ge\left\lceil\frac{2q^d-r^*}{q+1}\right\rceil.
\end{equation}
The question arises, whether $(\ref{elf})$ can be improved for all $\Sigma(\ffq,R)$ due to its algebraic
definition in terms of $\mathbb F_q$ and $R$.
The answer is negative: take the geometric model $Q^+(3,q)$ of $\Sigma(\ffq,\ffq\times\ffq)$.
Since each line of $Q^+(3,q)$ is a \BS, we see that (\ref{elf}) is sharp.
A further example of a \BS\ for which in (\ref{elf}) the equality holds will be dealt with in case (\itl i)
of theorem \ref{principale}.
For this reason we have to investigate \BS s in particular chain geometries.

The case in which $\Sigma(\ffq,R)$ is a M\"obius plane has been dealt with in \cite{BrRo85,Gl82,KiMaPa05}
(actually,  the results in these papers hold for arbitrary 3-$(q^2+1,q+1,1)$-designs).
In the quoted papers it is proved that if $B$ is a blocking set in a M\"obius plane of
order $q$, then
\begin{equation}\label{inv}
  \#B\geq 2q-1;
\end{equation}
furthermore, $\#B\geq2q$ for $q\geq4$.
Examples of blocking sets attaining the lower bounds are known only
for $q\leq5$ and were found and classified by means of a computer
search \cite{KiMaPa05}.
If more generally $R$ is a local ring we can give a generalization of (\ref{inv}) for sufficiently
large $q$.
To this end we use a polynomial of constant sign introduced in \cite{Gl82}.

Let $\theta_n=(q^{n+1}-1)/(q-1)$ for $n\in\mathbb N\cup\{-1\}$.
\begin{theorem}\label{principale}
  Let $B$ be a blocking set in the chain geometry $\Sigma(\ffq,R)$, where $R$ is a local ring,
  and let
  \begin{equation}\label{delta}
    d=\dim_{\mathbb F_q}R,\qquad \delta=\dim_{\mathbb F_q}(R\setminus R^*),
  \end{equation}
  where $R^*$ denotes the set of units of $R$.
  Then\\
  {\rm ({\em i\/})} if $\delta=d-1$, then $\#B\geq q^{d-1}$; the equation
  $\#B=q^{d-1}$ holds if, and only if, $B$ is a parallel class;\\
  {\rm ({\em ii\/})} if $d>2$, $\delta=d-2$ and $\varepsilon>0$, then
  $\#B>2q^{d-1}-(\frac72+\varepsilon)q^{d-2}$ for $q$ sufficiently
  large;\\
  {\rm ({\em iii\/})} if $d>2$, $\delta<d-2$ and $\varepsilon>0$, then
  $\#B>2q^{d-1}-(1+\varepsilon)q^{d-2}$ for $q$ sufficiently
  large.
\end{theorem}

\begin{proof}
For each point $p$, let $[p]$ denote the related parallel class.
We have $\#[p]=q^{\delta}$.

({\em i\/})
The first assertion follows from (\ref{elf}).

Now assume that $B$ is a \BS\  such that $\#B=q^{d-1}$.
Let $p$ be a point outside $B$.
Every chain through $p$ intersects $B$, so $B\setminus[p]$
contains at least $\lambda_1/\lambda_2=q^{d-1}$ points.
Therefore $[p]\cap B=\emptyset$.
This holds for any point not in $B$, so $B$ is a parallel class.


({\em ii\/}), ({\em iii\/})
Let $x=\#B$.
Denote by $n_i$ the number of chains meeting $B$ in exactly $i$
points. Since $B$ is a \BS, $n_0=0$.
By (\ref{ela}), taking into account $v=q^d+q^\delta$, we have
\begin{equation}\label{car1}
  \sum_{i\ge1}n_i=\lambda_0=\left(q^{2(d-\delta-1)}+q^{2(d-\delta-2)}+
  \cdots+q^2+1\right)q^{d+2\delta-1}\lambda_3.
\end{equation}
Computing in two ways the number of the ordered pairs $(p,C)$, $C$
a chain, $p\in B\cap C$, we obtain
\begin{equation}\label{car2}
  \sum_{i\ge1}in_i=x\lambda_1=xq^{d+\delta-1}\theta_{d-\delta-1}\lambda_3.
\end{equation}
Analogously, by taking into account the ordered triples $(p_1,p_2,C)$ and
quadruples $(p_1,p_2,p_3,C)$, where the $p_i$s are distinct points of $B$
incident with $C$, we have
\begin{equation}\label{car3}
  \sum_{i\ge1}i(i-1)n_i\ge x(x-q^{\delta})\lambda_2=
  x(x-q^{\delta})\theta_{d-\delta-1}q^{\delta}\lambda_3;
\end{equation}
\begin{equation}\label{car4}
  \sum_{i\ge1}i(i-1)(i-2)n_i\le x(x-1)(x-2)\lambda_3.
\end{equation}
The polynomial
\[
  P(i)=(i-1)(i-3)(i-4)=i(i-1)(i-2)-5i(i-1)+12i-12,
\]
introduced in \cite{Gl82}, is non-negative for all positive integers $i$. From
(\ref{car1})--(\ref{car4}), it follows
\begin{eqnarray}
  \mbox{{}\qquad}0&\le&\frac1{\lambda_3}\sum_{i\ge1}n_iP(i)\le\nonumber\\
  &\le&x(x-1)(x-2)-5x(x-q^{\delta})\theta_{d-\delta-1}q^{\delta}+
  12xq^{d+\delta-1}\theta_{d-\delta-1}\label{polinomione}\\
  &&-12\left(q^{2(d-\delta-1)}+q^{2(d-\delta-2)}+
  \cdots+q^2+1\right)q^{d+2\delta-1}.\nonumber
\end{eqnarray}
Assume
\begin{equation}\label{sostituz}
  x=2q^{d-1}-kq^{d-2}.
\end{equation}
Since $d>2$, by (\ref{polinomione}) and (\ref{sostituz}) we obtain
\begin{equation}\label{lterms}
  0\le(4-4k)q^{3d-4}+10q^{2d+\delta-2}+\mbox{ (terms of degree $<3d-4$)}
\end{equation}
and this implies ({\em ii\/}) and ({\em iii\/}).
\end{proof}

\begin{remark}
In the previous proof, actually
only the combinatorial structure of $\Sigma(\mathbb{F}_q,R)$ is
essential, and such structure is a
3-$(q^{\delta},q+1,\lambda_3)$-divisible design with
$q^d+q^{\delta}$ points.
See \cite{Ha04} for generalities on divisible designs.
\end{remark}
\begin{remark}
A proposition like ({\em i\/}),
characterizing some geometric configurations as blocking sets of
minimum size, is often called a {\em Bose-Burton type theorem.}
\end{remark}
\begin{remark}
If $K$ and $F$ are fields with $K\subseteq F$ and $[F:K]=d$, then $\Sigma(K,F)$ is
called a $d$-\itl{dimensional M\"obius geometry over $K$}.
Theorem \ref{principale} gives in particular a lower bound for the blocking sets in the
finite M\"obius geometries.
\end{remark}
\begin{remark}
In case the term of degree $3d-4$ in (\ref{lterms}) vanishes, the term of degree $3d-5$
always turns out to be positive, with one exception given by $d=3$ and $\delta=0$.
By substituting $x=2q^2-q+t$ in (\ref{polinomione}) we obtain
\begin{eqnarray}
  0&\le&(-1+4t)q^4+(19-10t)q^3+(-11-2t+t^2)q^2\\
  &&+(-7+21t-8t^2)q+(7t-8t^2+t^3),\nonumber
\end{eqnarray}
whence
\begin{theorem}
  Let $B$ be a \BS\ in the three-dimensional M\"obius geometry over $\ffq$.
  Then $\#B\ge 2q^2-q-2$.
  Furthermore, $\#B\ge2q^2-q-1$ for $q\ge4$, $\#B\ge2q^2-q$ for $q\ge7$, and $\#B\ge2q^2-q+1$
  for $q\ge19.$
\end{theorem}
\end{remark}

It is not clear whether there exist blocking sets of size near to
the lower bounds given in theorem \ref{principale}.
In \cite{GrRo02,Sz92} the existence of \BS s in the M\"obius planes of size $O(q\log q)$ is proved.
The following theorem allows to construct blocking sets in generalized chain geometries,
starting from \BS s in M\"obius geometries.
\begin{theorem}\label{ponte}
  Let $R$ be a local ring, and $F=R/(R\setminus R^*)$.
  If $\Sigma(\mathbb F_q,F)$ contains a \BS\ of size $x$, then
  $\Sigma(\mathbb F_q,R)$ contains a \BS\ of size $xq^\delta$ ($\delta$ as in (\ref{delta})).
\end{theorem}

\begin{proof}
  Let $I=R\setminus R^*$ and, for $R(a,b)\in\mathbb P(R)$,
  $R(a,b)^\varphi=F(a+I,b+I)$.
  We obtain a well-defined map $\varphi:\,\mathbb P(R)\rightarrow\mathbb P(F)$ such that
  (\itl i) if $C$ is a chain in $\Sigma(\mathbb F_q,R)$, then $C^\varphi$ is a chain in
  $\Sigma(\mathbb F_q,F)$, (\itl{ii}) for $p,q\in\mathbb P(R)$, it holds $p^\varphi=q^\varphi$
  if and only if $p\|q$. 
  By such properties, if $B$ is a \BS\ in $\Sigma(\mathbb F_q,F)$ with $\#B=x$, then $B=B_0^\varphi$,
  where $B_0$ can be chosen as the union of exactly $x$ parallel classes, each of size $q^\delta$.
  This $B_0$ is a \BS\  of size $xq^\delta$ in $\Sigma(\mathbb F_q,R)$.
\end{proof}

\begin{corollary}
  If $R$ is a local ring and $d=\dim_{\mathbb F_q}R=2+\delta$,
  then $\Sigma(\mathbb F_q,R)$ contains a \BS\ of size $O(q^{d-1}\log q)$.
\end{corollary}


\bibliographystyle{plain}

\end{document}